\documentclass[11pt,leqno]{amsart}

\usepackage{hyperref}
\usepackage{mathrsfs}
\usepackage{tikz-cd}
\usepackage{amsmath,amsfonts,amssymb}
\usepackage{xfrac}
\usepackage{nicefrac}
\usepackage{color}

\theoremstyle{plain}
\newtheorem{theorem}{Theorem}[section]
\newtheorem{lemma}[theorem]{Lemma}
\newtheorem{proposition}[theorem]{Proposition}

\theoremstyle{definition}

\newtheorem*{definition*}{Definition}

\DeclareMathOperator{\st}{St}
\DeclareMathOperator{\Aut}{Aut}
\DeclareMathOperator{\St}{St}

\DeclareMathOperator{\Z}{\mathbb{Z}}

\numberwithin{equation}{section}

\title[Ramification structures for quotients of multi-EGS groups]{Ramification structures for quotients\\ of multi-EGS groups}

 \author[E.\, Di Domenico]{Elena Di Domenico}
 
 \address{Elena Di Domenico: Department of Mathematics, University of Trento, 38123 Trento, Italy - University of the Basque Country UPV/EHU, 48080 Bilbao, Spain}
 \email{elena.didomenico@unitn.it}
 
 \author[\c{S}. G\"{u}l]{\c{S}\"{u}kran G\"{u}l}
 
\address{\c{S}\"{u}kran G\"{u}l: Department of Mathematics\\ TED University\\
06420 Ankara, Turkey} 
\email{sukran.gul@tedu.edu.tr}
 
 \author[A. Thillaisundaram]{Anitha Thillaisundaram}
 
 \address{Anitha Thillaisundaram: Centre for Mathematical Sciences, Lund University,  223 62 Lund, Sweden}
 \email{anitha.t@cantab.net}
 
 \thanks{The  first and the second author are supported by the Spanish Government, grant MTM2017-86802-P, partly with FEDER funds, and by the Basque Government, grant IT974-16. 
  The first author is also supported by the National Group for Algebraic and Geometric Structures, and their
Applications (GNSAGA - INdAM).
 The third author acknowledges support from EPSRC, grant EP/T005068/1, and from the Lincoln Institute of Advanced Studies. 
 She also thanks the University of the Basque Country for its hospitality.
The first and second authors would like to thank the School of Mathematics and Physics at the University of Lincoln for its hospitality while this research was conducted.
}

 \date{\today}
 
\begin{document}

\begin{abstract}
 Groups associated to surfaces isogenous to a higher product of curves can be characterised by a purely group-theoretic condition, which is the existence of a so-called ramification structure. 
 G\"{u}l and Uria-Albizuri showed that quotients of the periodic Grigorchuk-Gupta-Sidki groups, GGS-groups for short, admit ramification structures. 
 We extend their result by showing that quotients of generalisations of the GGS-groups, namely multi-EGS groups, also admit ramification structures.
\end{abstract}

\maketitle

\section{Introduction}

An algebraic surface $S$ is \emph{isogenous to a higher product of curves} if it is isomorphic to $(C_1\times C_2)/\Gamma$, where $C_1$ and $C_2$ are curves of genus at least~$2$, and $\Gamma$ is a finite group acting freely on $C_1\times C_2$. 
Such groups~$\Gamma$ 
can be characterised by a group-theoretic condition, which is the existence of a ramification structure:

\begin{definition*}
Let $\Gamma$ be a finite group and let $T=(g_1,g_2,\ldots,g_r)$ be a tuple of non-trivial elements of~$\Gamma$.

\begin{enumerate}
\item[$\bullet$] The tuple~$T$ is called a \emph{(spherical) system of generators} of~$\Gamma$ if 
\[
\langle g_1,\ldots,g_r\rangle =\Gamma\quad \text{and}\quad g_1g_2\cdots g_r=1.
\]
\item[$\bullet$] Let $\Sigma(T)$ be the union of all conjugates of the cyclic subgroups generated by the elements of~$T$, that is,
\[
\Sigma(T)=\bigcup_{g\in \Gamma} \bigcup_{i=1}^r \langle g_i\rangle^g.
\]
We say that two tuples~$T_1$ and~$T_2$ are \emph{disjoint} if $\Sigma(T_1)\cap \Sigma(T_2)=\{1\}$.
\item[$\bullet$] An \emph{unmixed ramification structure} of size $(r_1,r_2)$ for a finite group~$\Gamma$ is a pair $(T_1,T_2)$ of disjoint systems of generators of~$\Gamma$, where $|T_1|=r_1$ and $|T_2|=r_2$.
\end{enumerate}
\end{definition*}

We are interested only in the case when $r_1=r_2$, so we will simply say \emph{ramification structure} in the sequel.   
When $r_1=r_2=3$, one uses the term \emph{Beauville structure} instead of ramification structure. 
Also, a group admitting a Beauville structure is called a \emph{Beauville group}. The motivation for studying groups with Beauville structures, or more generally, ramification structures, comes  from the problem of classifying surfaces of complex dimension two isogenous to a higher product of two curves. We refer the reader to~\cite{NT} (and references therein) for a more detailed introduction to groups with ramification structures.

Now Beauville groups have gained much attention over the past decades; see for example the surveys \cite{ BBF, fai, fai2, jon}. 
In particular, the abelian Beauville groups were classified by Catanese \cite{cat}: a finite abelian group~$\Gamma$ is a Beauville group if and only if $\Gamma\cong C_n \times C_n$ for $n> 1$ with $\gcd(n, 6) = 1$. 
After abelian groups, it is natural to consider nilpotent groups. 
We note that the determination of nilpotent Beauville groups can be reduced to that of Beauville $p$-groups. 
We refer the reader to~\cite{GUA} for an overview of results concerning Beauville $p$-groups.

On the other hand, not much is known about ramification structures that are not Beauville.
Garion and Penegini have studied the existence of ramification structures for alternating and symmetric groups (which are also Beauville groups) in~\cite{GP}, and soon afterwards they characterised the abelian groups with ramification structures in~\cite{GP2}. 
Up until recently, the existence of ramification structures that are not Beauville for non-abelian $p$-groups have been only studied in~\cite{gul}.

Groups acting on rooted trees were first seen as a source for groups admitting ramification structures from the work of G\"{u}l and Uria-Albizuri. In~\cite{GUA}, they  showed that the quotients of periodic Grigorchuk-Gupta-Sidki (GGS-)groups admit Beauville  structures.
The GGS-groups  form an infinite family of groups acting on the $p$-adic tree, for $p$ an odd prime. 
 Briefly, a GGS-group is a group $G=\langle a,b\rangle$, where $a=(1\,2\, \cdots \,p)$ permutes the first-level vertices, whereas $b$ fixes the first-level vertices and is defined recursively by $b=(a^{e_1},\ldots,a^{e_{p-1}},b)$, for $e_1,\ldots,e_{p-1}\in \mathbb{F}_p$ with not all $e_i$ being zero. 
 We write $\mathbf{e}=(e_1,\ldots,e_{p-1})$ and call it the \emph{defining vector} of~$G$. We refer the reader to Sections~2 and 3 for  background material and definitions.

These groups first arose as easily describable examples of infinite finitely generated periodic groups, but  over the years groups acting on rooted trees have established themselves as a source of groups with  many other interesting properties, with links and applications to other areas of mathematics.

The result of~\cite{GUA} was extended in~\cite{DGT} to periodic GGS-groups acting on the $p^n$-adic tree, for $p$ any prime and $n\in\mathbb{N}$. 
Recently it was shown that quotients of the Grigorchuk groups admit non-Beauville ramification structures~\cite{NT}, where the  Grigorchuk groups are the earliest constructed infinite family  of groups acting on the binary rooted tree; see~\cite{Grigorchuk}.

In this paper, we show that  quotients of  multi-EGS groups  also admit ramification structures  that are not Beauville.
For $p$ an odd prime, a multi-EGS group acting on the $p$-adic tree is, loosely speaking, a group $G$ which is similar to a GGS-group acting on the $p$-adic tree, but with many $b$-type generators. 
We refer the reader to Section~3 for  the precise definition.

The multi-EGS groups are infinite groups acting faithfully on the $p$-adic tree. The natural  quotients of a multi-EGS group~$G$ are $G/\st_G(n)$, for $n\in\mathbb{N}$, where $\st_G(n)$ is the normal subgroup of the elements of~$G$ that pointwise fix the vertices at the $n$th level of the tree. 
The quotient $G/\st_G(n)$ acts on the finite tree consisting of the first $n$ layers of the full $p$-adic tree.
For $G$ a multi-EGS group, let $r$ be the dimension of the vector space spanned by all defining vectors of~$G$.
Then the main result of this paper is as follows.

\begin{theorem}\label{thm:multi-EGS}
Let $G$ be a multi-EGS group that is not a GGS-group. Then $G/\st_G(k)$ admits a ramification structure for
\begin{enumerate}
\item[(i)] $k\ge r+1$ if $G$ is regular branch over $G'$;
\item[(ii)] $k\ge 5$ if $G$ is periodic and not regular branch over  $G'$;
\item[(iii)] $k\ge 3$ if $G$ is non-periodic and not regular branch over  $G'$.
\end{enumerate}
\end{theorem}

Theorem~\ref{thm:multi-EGS} shows that a multi-EGS group is a source for the construction of an infinite series of $p$-groups admitting ramification structures for every odd prime~$p$.
We observe that the above result includes non-periodic multi-EGS groups that are not GGS-groups, whereas from \cite{GUA} we know that no quotient $G/\st_G(n)$, for $G$ a non-periodic GGS-group, admits a ramification structure.

\smallskip

\noindent\emph{Organisation.} Section~2 of this paper consists of background material for groups acting on rooted trees. 
In Section~3 we introduce the multi-EGS groups. In Section~4 we prove  Theorem~\ref{thm:multi-EGS}.

\smallskip

\noindent\emph{Notation.}
If $\Gamma$ is a finitely generated group, we write $d(\Gamma)$ for the minimum number of generators of~$\Gamma$.


\section{Preliminaries}

All trees considered here will be rooted, meaning that there is a distinguished vertex called the root, with one degree less than all other vertices. 
Throughout the paper, $p$ denotes an odd prime.
Let $\mathcal{T}$ be the $p$-adic  tree, meaning all vertices have $p$ children.  
Using the alphabet $X = \{1,\ldots,p\}$, the vertices $u_\omega$ of $\mathcal{T}$ are labelled bijectively by elements $\omega$ of the free
 monoid~$X^*$ as follows: the root of~$\mathcal{T}$
 is labelled by the empty word~$\varnothing$, and for each word
 $\omega \in X^*$ and letter $x \in X$ there is an edge connecting $u_\omega$ to~$u_{\omega x}$.  
 We say that $u_\omega$ precedes $u_\lambda$ whenever $\omega$ is a prefix of $\lambda$.

  A natural length function on~$X^*$ is defined as follows: the words
  $\omega$ of length $\lvert \omega \rvert = n$, representing vertices
  $u_\omega$ that are at distance $n$ from the root, are the $n$th-level vertices and form the \textit{$n$th layer} of the tree. 
  The elements of the boundary $\partial \mathcal{T}$ correspond
  naturally to infinite simple rooted paths, and they are in one-to-one correspondence with the $p$-adic integers.

  Denote by $\mathcal{T}_u$ the full  subtree of~$\mathcal{T}$ that has its root at the vertex~$u$ and includes all vertices succeeding~$u$.  For any two vertices $u = u_\omega$ and $v = u_\lambda$, the map
  $u_{\omega \tau} \mapsto u_{\lambda \tau}$, induced by replacing the
  prefix $\omega$ by $\lambda$, yields an isomorphism between the
  subtrees $\mathcal{T}_u$ and~$\mathcal{T}_v$.

 We note that every automorphism of $\mathcal{T}$ fixes the root, and the orbits of $\mathrm{Aut}(\mathcal{T})$ on the vertices of the tree~$\mathcal{T}$ are precisely its layers. 
 For $f \in \mathrm{Aut}(\mathcal{T})$, the image of a vertex $u$ under
 $f$ is denoted by~$u^f$.
 The automorphism $f$ induces a faithful action
 on the monoid~$X^*$ such that
 $(u_\omega)^f = u_{\omega^f}$. 
 For $\omega \in X^*$ and
 $x \in X$ we have $(\omega x)^f = \omega^f x'$ where $x' \in X$ is
 uniquely determined by $\omega$ and~$f$. 
 This induces a permutation
 $f(\omega)$ of $X$ so that
  \[
  (\omega x)^f = \omega^f x^{f(\omega)}, \qquad \text{and hence}
  \quad   (u_{\omega x})^f = u_{\omega^f x^{f(\omega)}}.
  \]
  The permutation~$f(\omega)$ is also called the \emph{label} of $f$ at $\omega$. The automorphism $f$ is \textit{rooted} if $f(\omega) = 1$ for all
  $\omega \ne \varnothing$.
  It is \textit{directed}, with directed
  path~$\ell$ for some  $\ell \in \partial \mathcal{T}$, if the support
  $\{ \omega \mid f(\omega) \ne 1 \}$ of its labelling is infinite and
  marks only vertices at distance~$1$ from the set of vertices corresponding  to the path~$\ell$.


\subsection{Subgroups of $\Aut(\mathcal{T})$}
 For $G\le \Aut(\mathcal{T})$, the
\textit{vertex stabiliser} $\St_G(u)$ is the subgroup
consisting of elements in $G$ that fix the vertex~$u$.
For
$n \in \mathbb{N}$, the \textit{$n$th-level stabiliser}
 $\St_G(n)= \bigcap_{\lvert \omega \rvert =n}
 \St_G(u_\omega)$
is the subgroup consisting of automorphisms that fix all vertices at
level~$n$. 
Let $\mathcal{T}_{[n]}$ be the finite subtree of~$\mathcal{T}$ of
vertices up to level~$n$. 
Then  $\St_G(n)$ is the kernel of the induced action of $G$ on $\mathcal{T}_{[n]}$, and we will denote by $G_n$ the quotient $G/\St_G(n)$.

Each $g\in \St_{\mathrm{Aut}(\mathcal{T})} (n)$ can be described completely in terms of its restrictions $g_1,\ldots,g_{p^n}$ to the subtrees rooted at vertices of level~$n$.
Indeed, the map 
\[
\psi_n \colon \St_{\mathrm{Aut}(\mathcal{T})}(n) \rightarrow
\prod\nolimits_{\lvert \omega \rvert = n} \mathrm{Aut}(\mathcal{T}_{u_\omega})
\cong \mathrm{Aut}(\mathcal{T}) \times \overset{p^n}{\cdots} \times
\mathrm{Aut}(\mathcal{T}),
\]
sending $g$ to $(g_1,\ldots,g_{p^n})$, is a natural isomorphism. For ease of notation, we write $\psi=\psi_1$. In addition,  we write 
\[
{\phi}_{n}:\text{St}_{G_n}(1)\rightarrow G_{n-1}\times \overset{p}\cdots\times G_{n-1}
\]
for the corresponding $\psi$ when working in the quotient.

For  $\omega\in X^*$, we further define 
\[
\varphi_\omega :\mathrm{St}_{\mathrm{Aut}(\mathcal{T})}(u_{\omega}) \rightarrow \mathrm{Aut}(\mathcal{T}_{u_\omega}) \cong \mathrm{Aut}(\mathcal{T})
\]
to be the natural restriction of $\mathrm{St}_{\mathrm{Aut}(\mathcal{T})}(u_{\omega})$ to  $\mathrm{Aut}(\mathcal{T}_{u_\omega})$.
A group $G \leq \Aut(\mathcal{T})$ is said to be \emph{self-similar} if the images under $\varphi_{\omega}$ and $\psi_n$ are contained in $G$ and $G \times \overset{p^n}{\cdots} \times G$, respectively, for all $\omega\in X^*$ and $n\in \mathbb{N}$.

We recall that a spherically transitive (that is, transitive on every layer of the tree) self-similar group $G$  with $ K \ne 1$ such that $K\times \overset{p}\cdots \times K \subseteq \psi( \st_K(1))$ and
$\lvert G : K \rvert < \infty$, 
is said to be \emph{regular
branch over $K$}. 
If, in the previous definition, we remove the condition $\lvert G : K \rvert < \infty$,  then $G$
is said to be \emph{weakly regular branch over~$K$}.


\section{Multi-EGS groups}

We first recall the GGS-groups acting on the $p$-adic tree~$\mathcal{T}$ and then proceed onto its generalisation.

\subsection{GGS-groups acting on the $p$-adic tree}

By $a$ we denote the rooted automorphism, corresponding to the
$p$-cycle $(1 \, 2 \, \cdots \, p)\in \text{Sym}(p)$, that cyclically
permutes the vertices $u_1, \ldots, u_p$ at the first level.
Recall
the coordinate map
\[
\psi\colon \st_{\Aut(\mathcal{T})}(1) \rightarrow \Aut(\mathcal{T}_{u_1}) \times
\cdots \times \Aut(\mathcal{T}_{u_p}) \cong \Aut(\mathcal{T}) \times \overset{p}{\cdots}
\times \Aut(\mathcal{T}).
\]
Given a  vector $
\mathbf{e} =(e_{1}, e_{2},\ldots , e_{p-1})\in
(\mathbb{F}_p)^{p-1}\backslash \{\mathbf{0}\}$,
we recursively define a  directed automorphism $b \in
\st_{\Aut(\mathcal{T})}(1)$ via
\[
\psi(b)=(a^{e_{1}}, a^{e_{2}},\ldots,a^{e_{p-1}},b).
\]
We call the subgroup $G=G_{\mathbf{e}}=\langle a, b
\rangle$ of $\Aut(\mathcal{T})$ the \textit{GGS-group} associated
to the defining vector $\mathbf{e}$.
We observe that $\langle a
\rangle \cong\langle b \rangle \cong C_p$
are cyclic groups of order~$p$, 
and $G$  acts spherically transitively on~$\mathcal{T}$.

If $G_{\mathbf{e}}=\langle a, b\rangle$ is a GGS-group corresponding to the defining vector $\mathbf{e}$, then $G$ is periodic, and thus an infinite $p$-group if and only if $ \sum\nolimits_{j=1}^{p-1} e_{j}\equiv 0 \pmod p$; compare \cite{NewHorizons, vov}.

We write $\mathcal{G}= \langle a,b \rangle$ with $\psi(b)=(a,\ldots,a,b)$,
for the GGS-group arising from the constant defining vector~$(1,\ldots,1)$. 
It is known that $\mathcal{G}$ is weakly regular branch \cite[Lem.~4.2]{FAZR2} but not branch \cite[Thm.~3.7]{FAGUA}.

\subsection{Multi-GGS groups acting on the $p$-adic tree}
Before defining multi-EGS groups, it is helpful to first define the subclass of multi-GGS groups, which more closely resemble GGS-groups.

Given $1\le r \le p-1$ and a finite $r$-tuple $\mathbf{E}$ of
$\mathbb{F}_p$-linearly independent vectors
\[
\mathbf{e}_i =(e_{i,1}, e_{i,2},\ldots , e_{i,p-1})\in
(\mathbb{F}_p)^{p-1}, \qquad i\in \{1,\ldots,r \},
\]
we recursively define directed automorphisms $b_1, \ldots, b_r \in
\st_{\Aut(\mathcal{T})}(1)$ via
\[
\psi(b_i)=(a^{e_{i,1}}, a^{e_{i,2}},\ldots,a^{e_{i,p-1}},b_i),
\qquad i\in \{1,\ldots,r \}.
\]
We call the subgroup $G=G_{\mathbf{E}}=\langle a, b_1, \ldots, b_r
\rangle$ of $\Aut(\mathcal{T})$ the \textit{multi-GGS group} associated
to the defining vectors $\mathbf{E}$.  
We observe that $\langle a\rangle \cong C_p$ and $\langle b_1, \ldots, b_r \rangle \cong C_p^r$ are elementary abelian $p$-groups, and we note that  a multi-GGS group $G$ acts spherically transitively on the tree~$\mathcal{T}$.

These groups were introduced in~\cite{AKT}, and were generalised in~\cite{KT} to multi-EGS groups, which are defined below. When $r=1$, then $G=\langle a,b_1\rangle$ is a GGS-group.

 From
 \cite[Prop.~4.3]{AKT}, we have that for $G=\langle a, b_1,\ldots,b_r\rangle$,
\begin{equation}\label{eq:abelianisation1}
G/G'=\langle aG', b_1G',\ldots,b_rG'\rangle\cong C_p^{\,r+1},
\end{equation}
and according to \cite[Prop.~6]{GUA2}, we have that
\begin{equation} \label{eq:inside G'}
\st_G(r+1) \leq G'.
\end{equation}

If $G_{\mathbf{E}}=\langle a, b_1, \ldots, b_r \rangle$ is a multi-GGS group corresponding to an $r$-tuple of defining vectors $\mathbf{E}=(\mathbf{e}_i)_{i=1}^r$, then $G$ is an infinite $p$-group if and only if for every $\mathbf{e}_i = (e_{i,1}, \ldots, e_{i,p-1})$, we have $ \sum\nolimits_{j=1}^{p-1} e_{i,j}\equiv 0 \pmod p$;
 compare \cite{NewHorizons, vov}.

For all multi-GGS groups $G \ne \mathcal{G}$, it follows from \cite[Prop.~3.5]{AKT}, that $G$ is regular branch over $\gamma_3(G)$. 
Furthermore, if $G$ is not a GGS-group, that is, if $G = \langle a, b_1, \ldots, b_r \rangle$ with $r\ge 2$, then $G$ is regular branch over $G'$ and $\psi(\st_G(1)')=G'\times \overset{p}{\cdots}\times G'$; compare~\cite[Lem.~2]{GUA2}.

On the other hand, for a multi-GGS group~$G$ the proof of \cite[Prop.~6]{GUA2} yields that $d(G_n)\geq n$ for $1\le n\le r+1$.
Indeed, if $G$ is regular branch over $G'$, then equality holds. 
More precisely, from \cite[Thm.~4.4]{FAGT} we have the following results:
For a multi-GGS group $G$ which is regular branch over $G'$, there exist so-called `distinguished directed automorphisms' $b_i$, for $1\leq i\leq r$, such that $G=\langle a, b_1, \ldots, b_r\rangle$, for $1\leq r\leq p-1$, and we have that
\begin{equation} \label{eqn: vanished bi}
\overline{b}_i\in G_n'  \quad \text{for all}  \ \ i\geq n  \ \ \text{ and for every}  \ \ n \in \{1,\ldots,r\},
\end{equation}
\begin{equation} \label{eqn: generators of Gn}
d(G_n)=n \ \ \text{and} \ \ 
G_n=\langle \overline{a},\overline {b}_1, \ldots, \overline{b}_{n-1} \rangle \ \ \text{for every}  \ \ n \in \{1,\ldots,r\};
\end{equation}
here we write $\overline{g}$ for the image of $g\in G$ in $G_n$.

Notice also that for any multi-GGS group $G$, we have that $d(G_n)=r+1$ for all $n\geq r+1$, by \eqref{eq:abelianisation1} and \eqref{eq:inside G'}.

\subsection{Multi-EGS groups acting on the $p$-adic tree}
We recall the definition of a multi-EGS group  from~\cite{KT}.
For $j \in \{ 1, \ldots, p \}$ let $r_j \in \{0,1,\ldots,p-1\}$, with
$r_j \not = 0$ for at least one index~$j$. 
We fix the numerical datum
$\mathbf{E} = (\mathbf{E}^{(1)}, \ldots, \mathbf{E}^{(p)})$, where
each
$\mathbf{E}^{(j)} = (\mathbf{e}^{(j)}_1, \ldots,
\mathbf{e}^{(j)}_{r_j})$
is an $r_j$-tuple of $\mathbb{F}_p$-linearly independent
vectors
\[
\mathbf{e}^{(j)}_i = \big( e^{(j)}_{i,1}, \ldots, e^{(j)}_{i,p-1}
\big) \in (\mathbb{F}_p)^{p-1}, \quad i \in \{1, \ldots, r_j
\}.
\]
Write $r=r_1+\cdots +r_p$, and we let $\mathbf{V}$ be the vector space spanned by the $r$ vectors in~$\mathbf{E}$.

The \emph{multi-EGS group} 
 associated to $\mathbf{E}$ is the group
\begin{align*}
 G = G_\mathbf{E} & = \langle a, \mathbf{b}^{(1)}, \ldots,
 \mathbf{b}^{(p)} \rangle \\
 & = \big\langle \{a \} \cup \{ b^{(j)}_i \mid 1 \leq j \leq p, \, 1
 \leq i \leq r_j \} \big\rangle
\end{align*}
where, for each $j \in \{1,\ldots,p\}$, the generator family
$\mathbf{b}^{(j)} = \{b^{(j)}_1, \ldots, b^{(j)}_{r_j}\}$ consists of
commuting directed automorphisms $b^{(j)}_i \in \mathrm{St}_{\Aut(\mathcal{T})}(1)$
along the directed path
\[
\big( \varnothing, (p-j+1), (p-j+1)(p-j+1), \ldots \big) \in \partial \mathcal{T}
\]
that satisfy the recursive relations
\[
\psi(b^{(j)}_i) = \Big( a^{e^{(j)}_{i,j}}, \ldots, a^{e^{(j)}_{i,p-1}},b^{(j)}_i,
a^{e^{(j)}_{i,1}},\ldots, a^{e^{(j)}_{i,j-1}} \Big).
\]
The vector $\mathbf{e}^{(j)}_i$ is the defining vector of
$b^{(j)}_i$. 
We say that $\mathbf{e}^{(j)}_i$ is \emph{symmetric} if $e^{(j)}_{i,k}=e^{(j)}_{i,p-k}$ for all $1\le k\le \frac{p-1}{2}$; otherwise $\mathbf{e}^{(j)}_i$ is \emph{non-symmetric}.

The multi-EGS groups  are infinite and act spherically transitively on~$\mathcal{T}$. 
From \cite[Prop.~3.9]{KT},  for $G= \langle a, \mathbf{b}^{(1)}, \ldots,
\mathbf{b}^{(p)} \rangle$,
\begin{equation*}\label{eq:abelianisation2}
G/G'=\langle aG', b_1^{(1)}G',\ldots,b_{r_1}^{(1)}G',\,\ldots\,, b_1^{(p)}G',\ldots,b_{r_p}^{(p)}G'\rangle\cong C_p^{\,r+1}.
\end{equation*}

A multi-EGS group~$G=G_\mathbf{E}$ is a periodic group if and only if for every $\mathbf{e} = (e_{1}, \ldots, e_{p-1})\in \mathbf{E}$, we have $ \sum\nolimits_{j=1}^{p-1} e_{j}\equiv 0 \pmod p$; compare~\cite[Lem.~3.13]{TUA}.

For $G=G_{\mathbf{E}}=\langle a, \mathbf{b}^{(1)}, \ldots,
\mathbf{b}^{(p)} \rangle$ a multi-EGS group, if either at least one $\mathbf{e}\in\mathbf{E}$ is non-symmetric, or $\dim \mathbf{V}\ge 2$, then $G$ is regular branch over $G'$. 
Else, a multi-EGS group $G\ne\mathcal{G}$ is regular branch over~$\gamma_3(G)$ if, for 
$j \in \{1, \ldots,p\}$, every non-empty family $\mathbf{b}^{(j)}=\{b_1^{(j)}\}$
consists of a single directed automorphism with $\mathbf{e}_1^{(j)}=\mathbf{e}$, where $\mathbf{e}$ is a fixed non-constant symmetric defining vector;
compare~\cite[Prop.~3.2]{TUA}.

We record the following result:

\begin{lemma}\label{lem:multi-GGS}
Let $G=\langle a, \mathbf{b}^{(1)}, \ldots,
\mathbf{b}^{(p)} \rangle$ be a multi-EGS group, let  $\ell\in\{1,\ldots,p\}$ and suppose $1\ne d\in \langle \mathbf{b}^{(\ell)}\rangle$ has a defining vector of zero sum.  
Then  the order of  $a^jd$ for  $j\in\mathbb{F}_p^*$ is $p^2$ in both $G$ and~$G_{3}$.
\end{lemma}

\begin{proof}
Write $(e_1,\ldots,e_{p-1})$ for the defining vector of~$d$. 
Then
\[
\begin{split}
(a^{j}d)^p&=d^{a^{j(p-1)}}\cdots d^{a^{2j}}d^{a^j}d
\end{split}
\]
and its image under $\psi$ contains components of the form
\[
a^{e_{k j}+e_{(k+1)j}+\cdots+e_{(p-1)j}}da^{e_j+e_{2j}+\cdots+e_{(k-1)j}}
\] 
for some $k\in\mathbb{F}_p^*$.
Since $\sum_{\lambda=1}^pe_{\lambda}\equiv  0\pmod p$ it follows that the components of $\psi((a^{j}d)^p)$ are the conjugates 
\[
d,d^{a^{e_{j}}}, d^{a^{e_j+e_{2j}}},\ldots, d^{a^{e_j+e_{2j}+\cdots+e_{(p-1)j}}}
\]
in a certain order, each of which has order~$p$, hence the order of $a^{j}d$ is $p^2$ for all $j\in\mathbb{F}_p^*$. 

Moreover since the components of  $\psi((a^{j}d)^p)$ are conjugates of $d$ that belong to $\st_G(1)\backslash \st_G(2)$ it follows that $(a^{j}d)^p\not\in \st_G(3)$ and its order in $G_3$ is $p^2$.
\end{proof}


\section{Ramification structures for multi-EGS groups}

In this section, we give the proof of Theorem~\ref{thm:multi-EGS}.
To this purpose, we will first prove the following general lemma, which is the generalisation of  \cite[Lem.~2.4]{gul2}.

\begin{lemma}\label{order preserve}
Let $\Gamma=\langle x,y_1, y_2, \ldots, y_{n-1} \rangle$ be an $n$-generator $p$-group, for $n\ge 2$ and $p$ a prime, and suppose that 
$\Gamma/\Gamma'=\langle x\Gamma'\rangle \times \langle y_1\Gamma'\rangle \times \langle y_2\Gamma'\rangle\times \cdots \times \langle y_{n-1}\Gamma'\rangle$. 
If $o(x)=o(x\Gamma')$ then
\[
\Big(\bigcup_{g\in \Gamma} {\langle x\rangle}^g  \Big)
\bigcap
\Big(\bigcup_{g\in \Gamma} {\langle y_i\rangle}^g \Big)= \{1\}
\]
for any $ i \in \{1, \ldots, n-1\}$.
\end{lemma}

\begin{proof}
Let $a=(x^j)^g=(y_i^k)^h$ for some $g,h\in \Gamma$, $i\in\{1,\ldots,n-1\}$ and $j, k \in \Z$. 
Then in $\overline{\Gamma}= \Gamma/\Gamma'$, we have that
$\overline{x}^{j}= \overline{y_i}^k$,
and the assumption on~$\overline{\Gamma}$ implies that 
$\overline{x}^j=\overline{1}$.
Thus $o(x\Gamma')$ divides $j$. 
Then the condition $o(x)=o(x\Gamma')$ yields that $a=1$, as desired.
\end{proof}

Now by~\cite{FAGT}, for multi-EGS groups and  multi-GGS groups which are regular branch over their respective derived subgroups, we have that level-stabiliser quotients of a multi-EGS group are isomorphic to level-stabiliser quotients of a multi-GGS group. Hence it suffices to consider multi-GGS groups in this case.

\begin{proposition}\label{words}
Let $G=\langle a, b_1,\ldots,b_r\rangle$, for $1\le r\le p-1$, be a multi-GGS group which is regular branch over $G'$ and let $b_1, \ldots, b_r$ be distinguished directed automorphisms of $G$.
Let $d_1,d_2\in \langle  b_1,\ldots,b_r\rangle\backslash\{1\}$ have defining vectors of zero sum in $\mathbb{F}_p$ such that
$d_1=b_1^{k_1}c_1$ and $d_2=b_1^{k_2}c_2$ where $0\leq k_1, k_2 \leq p-1$ and $c_1, c_2 \in \langle  b_2,\ldots,b_r\rangle$.
If in~$G_3$, we have
\[
\langle ( ad_1)^{p}\rangle=\langle (ad_2)^{p}\rangle^g
\]
with $g\in G_{3}$, then $k_1=k_2$.
\end{proposition}

\begin{proof}
We will mainly follow  the strategy in \cite[Prop.~3.4]{GUA}. 
Writing $w_{1}=(ad_1)^p$ and $w_{2}=(ad_2)^p$, the equality
$\langle (ad_1)^{p}\rangle=\langle (ad_2)^{p}\rangle^g$ implies that 
\begin{equation}\label{word-clash}
w_1^{\,\mu}= w_2^{\,g}
\end{equation}
for some $1\le \mu\le p-1$. 

Writing $(e_1,\ldots,e_{p-1})$ and $(f_1,\ldots,f_{p-1})$ for the defining vectors of $d_1$ and $d_2$ respectively, we have
\[
\phi_3(w_{1})=\big(d_1^{\,a^{e_1}}, d_1^{\,a^{e_1+e_2}},\ldots,d_1^{\,a^{e_1+\cdots+e_{p-2}}},d_1, d_1\big).
\]
As $\text{St}_{G_2}(1)$ is abelian, the components of $\phi_3(w_{2}^g)$ are of the form $d_2^{\,a^m}$, with $m\in\mathbb{F}_p$. 
Furthermore, one of the components of $\phi_3(w_{1}^\mu )$ is $d_1^{\,\mu}$. It then follows from~\eqref{word-clash} that  $d_2^{\,a^m}=d_1^{\,\mu}$ in~$G_2$ for some $m$. 

Write $g=a^sh_s$ for some $s\in\mathbb{F}_p$ and $h_s\in \st_{G_{3}}(1)$. 
We have
\[
\phi_3(w_2^{\,a^s})=\big(d_2^{\,a^{f_1+\cdots +f_{p-(s-1)}}},\ldots, d_2,d_2,\ldots,
 d_2^{\,a^{f_1+\cdots+f_{p-s}}}\big),
\]
where  the first $d_2$ occurs at the $(s-1)$st component. 
Then
 \[
 \phi_3(w_2^{a^s}) =\big(d_1^{\,\mu a^{-m+(f_1+\cdots +f_{p-(s-1)})}},\ldots, d_1^{\,\mu a^{-m}},d_1^{\,\mu a^{-m}},
\ldots,
 d_1^{\,\mu a^{-m+(f_1+\cdots+f_{p-s})}}\big).
 \]
 From \eqref{word-clash}, we obtain
 \begin{align*}
 \phi_3(h_s)&=(a^{e_1+m-(f_1+\cdots +f_{p-(s-1)})}u_1, \ldots,
 a^{(e_1+\cdots+e_{p-2})+m-(f_1+\cdots +f_{p-s-2})}u_{p-2},\\
 &\qquad\qquad\qquad\qquad\qquad\qquad a^{m-(f_1+\cdots +f_{p-s-1})}u_{p-1} ,a^{m-(f_1+\cdots +f_{p-s})}u_p ),
 \end{align*}
 where $u_1,\ldots,u_p\in \st_{G_2}(1)$.

Recursively defining $h_{i-1}=h_i(d_2^{\,-1})^{a^{i-1}}$ for $i=s,\ldots,1$, we get
\[
\phi_3(h_0)=(a^{m+e_1-f_1}v_1, a^{m+(e_1+e_2)-(f_1+f_2)}v_2,\ldots, a^mv_{p-1},a^mv_p )
\]
for $v_1,\ldots,v_p\in \st_{G_2}(1)$.

We observe that 
\[
\phi_3(h_0^{\,a})\equiv \phi_3(h_0d_2d_1^{\,-1})\quad \pmod {\st_{G_2}(1)\times\overset{p}\cdots\times \st_{G_2}(1)}.
\]
Hence 
\[
\phi_3(d_1d_2^{\,-1}[h_0,a])\in \st_{G_2}(1)\times\overset{p}\cdots\times \st_{G_2}(1)\,\cap\, \phi_3(\st_{G_3}(1)) = \phi_3(\st_{G_3}(2)).
\]
Therefore $d_1d_2^{\,-1}=[a,h_0]$ in $G_2$.
Hence $d_1d_2^{\,-1}=b_1^{k_1-k_2}c_1c_2^{-1}\in G_2'$. 
By \eqref{eqn: vanished bi}, we know that $c_1c_2^{-1} \in G_2'$.
This implies that $b_1^{k_1-k_2} \in G_2'$.
Then the result follows from~\eqref{eqn: generators of Gn}. 
\end{proof}

Recall that for a $d$-generator finite group~$\Gamma$, the generating sets $\{x_1,\ldots, x_d\}$ and $\{y_1,\ldots,y_d\}$ form a ramification structure for~$\Gamma$, if and only if
\begin{equation}\label{intersection1}
\langle x\rangle \cap \langle y^g\rangle =\{1\}
\end{equation}
for all $x\in T_1=\{x_1,\ldots,x_d,x_1\cdots x_d\}$, $y\in T_2=\{y_1,\ldots,y_d,y_1\cdots y_d\}$ and $g\in \Gamma$.

\vspace{10pt}

We are now ready to prove part (i) of Theorem~\ref{thm:multi-EGS}.

\begin{theorem} \label{multi-GGS main}
Let $G=\langle a,b_1,b_2,\ldots, b_{r}\rangle$, for $2\le r\le p-1$, be a multi-GGS group. 
Then the quotient $G/\st_G(k)$, for $k>r$, admits a ramification structure.
\end{theorem}

\begin{proof}
Notice that as $r\geq 2$, the group~$G$ is regular branch over~$G'$.
Observe also that $b_i$ commutes with $b_j$ for any $i ,j \in \{1, \ldots, r\}$. 
Thus $b_i^{-1}b_j$ and $b_1b_2\cdots b_r$ are  of order~$p$.
For fixed $k>r$, we deal separately with two cases where $G$ is periodic and where $G$ is non-periodic.

Let us first assume that $G$ is periodic. 
In this case we will assume that the $b_i$'s are distinguished directed automorphisms for $2\leq i\leq r$.
Note that 
\[
\{a^{-1},ab_1,b_2,\ldots,b_r\} \quad\text{and}\quad \{ab_1^{\,2}, b_1^{-2}, b_1^{-1}b_2, b_2^{-1}b_3,\ldots, b_{r-1}^{-1}b_r \}
\]
are both systems of generators of $G_k=G/\text{St}_G(k)$. 
We claim that they yield a ramification structure for $G_k$. 
Setting 
\begin{align*}
T_1&=\{a^{-1},ab_1,b_2,\ldots,b_r, b_1b_2\cdots b_r\}\\
T_2&=\{ ab_1^{\,2}, b_1^{-2}, b_1^{-1}b_2, b_2^{-1}b_3,\ldots, b_{r-1}^{-1}b_r, ab_1^{\,-1}b_r\},
\end{align*}
we have to see that \eqref{intersection1} holds.

Let $x=a^{-1} \in T_1$, which is an element of order $p$. If $y=ab_1^{\,2}$,  it suffices to show that $\langle a^{-1}\rangle \ne \langle  (ab_1^{\,2})^p\rangle^g$ for all $g\in G_k$. 
As $a^{-1}\not\in \st_{G_k}(1)$ but $(ab_1^{\,2})^p\in \st_{G_k}(1)$, the claim~\eqref{intersection1} holds for this case. 
The same argument holds for  all other $y\in T_2$. 

Next suppose that $x=ab_1$. 
For $y\in\{ab_1^{\,2}, ab_1^{\, -1}b_r\}$,  we apply Proposition~\ref{words}  to conclude that~\eqref{intersection1} holds. 
So suppose that $y\in T_2\cap \langle b_1,\ldots,b_r\rangle$. 
Note that
\[
\phi_k\big(((ab_1)^p)^g\big)=\big( b_1^{\,g_1},\ldots,b_1^{\,g_p} \big)
\]
for some $g_1,\ldots,g_p\in G_{k-1}$, whereas
\[
\phi_k(y)=(a^*,\ldots,a^*,y),
\]
where $*$ are unspecified exponents. 
Thus, since not every component of $\phi_k(y)$ is in $\st_{G_{k-1}}(1)$, the claim follows. 
Similarly for the case $x\in T_1\cap \langle b_1,\ldots,b_r\rangle$ with $y\in\{ab_1^{\,2}, ab_1^{\, -1}b_r\}$.

Lastly, we consider the cases $x\in T_1\cap \langle b_1,\ldots,b_r\rangle$ with $y\in T_2\cap \langle b_1,\ldots,b_r\rangle$.
Taking into account \eqref{eq:abelianisation1} and \eqref{eq:inside G'} together with the fact that the orders of $x$ and $y$ are $p$ in  both $G_k$ and $G_k/G_k'$, the claim follows from Lemma~\ref{order preserve}.

We now suppose that $G$ is not periodic.
Then there is at least one defining vector whose exponent sum is not  congruent to $0$ modulo $p$.
By multiplying some of the directed automorphisms $b_1, \ldots, b_r$, if necessary, we may assume that, the defining vector $\mathbf{e}_1$ of $b_1$ has exponent sum not congruent to 0 modulo $p$, and the defining vectors of $b_2,\ldots, b_r$ have zero exponent sum, modulo $p$.

Note that
\[
\{a^{-1}, ab_1, b_2, b_3, \ldots, b_r\}
 \quad\text{and}\quad 
\{ab_2,b_1, b_1^{\, -1}b_2, b_2^{\,-1}b_3, \ldots, b_{r-1}^{\,-1}b_r\}
\]
are both systems of generators of~$G_k$. 
Set
\begin{align*}
T_1&=\{a^{-1},ab_1,b_2,\ldots,b_r, b_1b_2\cdots b_r\}\\
T_2&=\{ ab_2,b_1, b_1^{\, -1}b_2, b_2^{\,-1}b_3,\ldots, b_{r-1}^{\,-1}b_r, ab_2b_r\}.
\end{align*}

If $x\in T_1\setminus \{ab_1\}$ and $y\in T_2$, then we apply Lemma~\ref{order preserve},  using  \eqref{eq:abelianisation1} and \eqref{eq:inside G'} and the fact that $x$ is of order~$p$.
The same argument applies if $x\in T_1$ with  $y\in T_2\cap \langle b_1,\ldots,b_r\rangle$.

Hence, it remains to consider the cases  $x=ab_1$  with $y \in \{ab_2, ab_2b_r \} \subseteq T_2$.   
We will only consider the case $y=ab_2$, as
for the case $y=ab_2b_r$, the same argument applies.
Recall that
\[
\phi_k\big(((ab_2)^p)^g\big)=\big(b_2^{\,g_1},\ldots,b_2^{\,g_p}\big)
\]
for  some $g_1,\ldots,g_p\in G_{k-1}$. 
However, writing $s\in\mathbb{F}_p^*$ for the sum of the components of $\mathbf{e}_1$,
\[
\phi_k((ab_1)^{p})=\big((a^sb_1)^{\,a^*},\ldots,(a^sb_1)^{\,a^*},a^sb_1\big),
\]
where $*$ are unspecified exponents of $a$. 
Already from comparing the second-level decomposition of $(ab_1)^{p^2}$ with the second-level decomposition of $(ab_2)^p$, we see that $\langle (ab_2)^p\rangle^g$, for any $g\in G_k$, has trivial intersection with $\langle (ab_1)^{p^{k-1}}\rangle$.

The proof is now complete.
\end{proof}

Next we consider multi-EGS groups~$G$ that are not regular branch over~$G'$. 
As mentioned towards the end of Section~3, these groups are defined by a single symmetric vector. 
In this case, for simplicity, we will write $b^{(j)}$ instead of $b_1^{(j)}$ in every non-empty family $\mathbf{b}^{(j)}=\{b_1^{(j)}\}$.

\begin{proposition}\label{EGS-key}
For $J\subseteq\{1,\ldots,p\}$, let $G=\langle a, b^{(j)} \mid j\in J\rangle$ be a periodic multi-EGS group defined by a single vector. Suppose $j,k\in J$ and  let $\lambda,\mu \in \mathbb{F}_p^*$ be such that $b:=(b^{(j)})^\lambda$ and $c:=(b^{(k)})^\mu$. If in~$G_3$ we have
\[
\langle (ab)^{p}\rangle=\langle (ac)^{p}\rangle^g
\]
with $g\in G_{3}$, then $\lambda=\mu$.
\end{proposition}

\begin{proof}
The proof is essentially the same as that of Proposition~\ref{words}, however due to the more complex notation, we give a full proof for convenience.  Writing $w_{1}=(a b)^p$ and $w_{2}=(a c)^p$, from the equation  $\langle (ab)^{p}\rangle=\langle (ac)^{p}\rangle^g$ we have
\begin{equation}\label{EGS-word-clash}
w_1^{\,\ell}=w_2^{\,g}
\end{equation}
for some $1\leq \ell \leq p-1$. 

We write $(e_1,\ldots,e_{p-1})$ for the defining vector of the group. Then
\[
\phi_3(w_{1})=\big(b^{a^{\lambda(e_1+\cdots+e_{j})}}, \ldots,b^{a^{\lambda(e_1+\cdots+e_{p-2})}},b, b,b^{a^{\lambda e_1}},\ldots , b^{a^{\lambda(e_1+\cdots+e_{j-1})}}\big)
\]
where the first $b$ appears in the $(p-j)$th component. 

As seen before, the components of $\phi_3(w_{2}^g)$ are of the form $c^{a^m}$, with $m\in\mathbb{F}_p$, and one of the components of $\phi_3(w_{1}^\ell )$ is $b^{\ell}$. By~\eqref{EGS-word-clash}, we have  $c^{a^m}=b^{\ell}$ in $G_2$ for some $m$.

Similarly we write $g=a^{s+(k-j)}h_s$ for some $s\in\mathbb{F}_p$ and $h_s\in \st_{G_{3}}(1)$.  Then
\[
\phi_3(w_2^{a^{s+(k-j)}})=(c^{a^{\mu(e_1+\cdots +e_{j-s})}},\ldots, c,c,\ldots, c^{a^{\mu(e_1+\cdots+e_{j-s-1})}}),
\]
where  the first $c$ occurs at the $(p-j+s)$th component, and
 \begin{align*}
 \phi_3(w_2^{a^{s+(k-j)}}) &=(b^{\ell a^{-m+\mu(e_1+\cdots +e_{j-s})}},\ldots, b^{\ell a^{-m}},b^{\ell a^{-m}}, \ldots,b^{\ell a^{-m+\mu(e_1+\cdots+e_{j-s-1}})}).
 \end{align*}
 From (\ref{EGS-word-clash}), we obtain
 \begin{align*}
 \phi_3(h_s)&=(a^{\lambda (e_1+\cdots +e_j)+m-\mu(e_1+\cdots +e_{j-s})}u_1, \ldots, a^{\lambda(e_1+\cdots +e_{p-2}) +m-\mu(e_1+\cdots+e_{p-s-2})}u_{p-j-1},\\
 &\,\,a^{m-\mu(e_1+\cdots +e_{p-s-1})}u_{p-j},a^{m-\mu(e_1+\cdots +e_{p-s})}u_{p-j+1}, a^{\lambda e_1 +m-\mu(e_1+\cdots +e_{p-s+1})}u_{p-j+2}, \\
 &\,\ldots , a^{\lambda (e_1+\cdots + e_{s-2}) +m-\mu(e_1+\cdots +e_{p-2})}u_{p-j+s-1}, a^{\lambda (e_1+\cdots + e_{s-1}) +m}u_{p-j+s},  \\
& a^{\lambda (e_1+\cdots + e_{s}) +m}u_{p-j+s+1}, a^{\lambda (e_1+\cdots + e_{s+1}) +m-\mu e_1}u_{p-j+s+2}, \ldots, \\
 & a^{\lambda(e_1+\cdots +e_{j-2})+m-\mu(e_1+\cdots +e_{j-s-2})}u_{p-1} ,a^{\lambda(e_1+\cdots +e_{j-1})+m-\mu(e_1+\cdots +e_{j-s-1})}u_p ),
 \end{align*}
 with $u_1,\ldots,u_p\in \st_{G_2}(1)$.

Recursively defining $h_{i-1}=h_i(c^{-1})^{a^{k-j+i-1}}$ for $i=s,\ldots,1$ yields
\begin{align*}
\phi_3(h_0)&=(a^{m+(\lambda -\mu) (e_1+\cdots +e_j)}v_1, a^{m+(\lambda -\mu)(e_1+\cdots + e_{j+1})}v_2,\ldots, a^mv_{p-j},a^mv_{p-j+1}\\
&\qquad\qquad\qquad\qquad a^{m+(\lambda-\mu)e_1}v_{p-j+2}, \ldots,  a^{m+(\lambda-\mu)(e_1+\cdots +e_{j-1})}v_p )
\end{align*}
for $v_1,\ldots,v_p\in \st_{G_2}(1)$.

From
\[
\phi_3(h_0^{\,a})\equiv \phi_3(h_0c^{a^{k-j}}b^{-1})\quad \pmod {\st_{G_2}(1)\times\overset{p}\cdots\times \st_{G_2}(1)},
\]
we obtain
\[
\phi_3\big(bc^{-a^{k-j}}[h_0,a]\big)\in \st_{G_2}(1)\times\overset{p}\cdots\times \st_{G_2}(1)\,\cap\, \phi_3(\st_{G_3}(1)) = \phi_3(\st_{G_3}(2)).
\]
Therefore $bc^{-a^{k-j}}\in G_2'$, and as $(b^{(k)})^{a^{k-j}} \equiv b^{(j)}$ mod $\st_G(2)$, the result follows.
\end{proof}

\begin{lemma} \label{EGS-single symmetric vector}
Let $G=\langle a,b^{(j_1)},b^{(j_2)},\ldots,b^{(j_n)}\rangle$, for some $2 \leq n \leq p$, be a periodic multi-EGS group defined by a single symmetric vector. Let 
\[
 x\in
 \big\{b^{(j_i)}b^{(j_{i+1})}, (b^{(j_i)})^{-1}b^{(j_{i+1})}, (b^{(j_i)})^{-1}(b^{(j_{i+1})})^{-1} \mid 1 \leq i \leq n-1\big\}.
\]
Suppose the order of $x$ in~$G$ is not~$p$. Then the order of $x$ is $p^2$ in both $G$ and $G_4$.
Furthermore $x^p \in \St_G(3)\backslash \St_G(4)$.
\end{lemma}

\begin{proof}
Let us write $\mathbf{e}=(e_1,\ldots,e_{p-1})$ for the defining vector of~$G$. 
We will show the result for $x=(b^{(j_i)})^{-1}b^{(j_{i+1})}$. 
The other cases follow similarly. Without loss of generality, we suppose that $j_{i+1}>j_i$. For simplicity, we write $j:=j_i$ and $k:=j_{i+1}$. Now
\begin{align*}
\psi(x)&=\psi\big((b^{(j)})^{-1}b^{(k)}\big)\\
&=\big( a^{-e_j}, \ldots, a^{-e_{p-1}}, (b^{(j)})^{-1},  a^{-e_1}, \ldots, a^{-e_{j-1}}\big)
\big( a^{e_k}, \ldots, a^{e_{p-1}}, b^{(k)},  a^{e_1}, \ldots, a^{e_{k-1}}\big)\\
&
=\big( a^{e_k-e_j}, \ldots, a^{e_{p-1}-e_{j+p-k-1}},a^{-e_{j+p-k}}b^{(k)}, a^{e_1-e_{j+p-k+1}},\ldots ,a^{e_{k-j-1}-e_{p-1}},\\
&\qquad\qquad\qquad\qquad\qquad\qquad\qquad\qquad\quad (b^{(j)})^{-1}a^{e_{k-j}},  a^{e_{k-j+1}-e_1}, \ldots, a^{e_{k-1}-e_{j-1}}\big).
\end{align*}
Then, since $e_{p-(k-j)}=e_{k-j}$,
\[
\psi(x^p)=\big(1, \ldots, 1, (a^{-e_{k-j}}b^{(k)})^p, \ 1, \ldots, 1,\  \big( (a^{e_{k-j}}(b^{(j)})^{-1})^p \big)^{a^{e_{k-j}}}, 1, \ldots, 1  \big).
\]
If $e_{k-j}=0$, then $x$ has order~$p$. Otherwise, analogous to  Lemma~\ref{lem:multi-GGS}, we have that $ (a^{-e_{k-j}}b^{(k)})^p,
(a^{e_{k-j}}(b^{(j)})^{-1})^p \in \St_G(2)\backslash \St_G(3)$, and they are of order~$p$.
It then follows that $x$ is of order~$p^2$ and $x^p \in\St_G(3)\backslash \St_G(4)$.
\end{proof}

\begin{theorem}\label{2-multi-EGS}
Let $G=\langle a,b^{(j_1)},b^{(j_2)},\ldots,b^{(j_n)}\rangle$, for some $2\leq n\leq p$, be a  multi-EGS group defined by a single symmetric vector. 
Then the quotient $G/\st_G(k)$ admits a ramification structure, for all $k\ge 5$ if $G$ is periodic, and for all $k\ge 3$ otherwise.
\end{theorem}

\begin{proof}
For simplicity, throughout the proof we will write $b_i$ instead of $b^{(j_i)}$  for all $1\leq i \leq n$. Also, without loss of generality, we suppose that $j_{i+1}>j_i$ for  all $1\leq i < n$.
Write $(e_1,\ldots,e_{p-1})$ for the defining vector of~$G$.
We deal separately with two cases where $G$ is periodic and where $G$ is non-periodic.

Let us first assume that $G$ is periodic. 
Then the defining vector~$\mathbf{e}$ is non-constant.
Also note that since $\mathbf{e}$ is symmetric, if $p=3$, then $\mathbf{e}$ is necessarily constant.
Thus, for the periodic case, we have that $p\geq 5$. 

We fix $k\geq 5$.
Then the sets
\[
\{ab_1,\ b_1^{-1}b_2,\ b_2^{-1}b_3,\ \ldots\ ,\ b_{n-1}^{-1}b_n, \ b_n^{-1}a\}
\]
and 
\[
\{b_1,\  b_1^{-1}b_2^{-1},\  b_2b_3,\ b_3^{-1}b_4^{-1},\ \ldots\ ,\ b_{n-2}b_{n-1},\ b_{n-1}^{-1}b_n^{-1},\ b_n^3a\}
\]
are both systems of generators for $G_k$, if $n$ is even.
If $n$ is odd, then instead of the latter set we simply consider the set
\[
\{b_1^{-1}, \ b_1b_2,\ b_2^{-1}b_3^{-1},\ \ldots\ ,\ b_{n-2}b_{n-1},\ b_{n-1}^{-1}b_n^{-1},\ b_n^3a\}.
\] 
Setting
\[
T_1= 
\{ab_1,\ b_1^{-1}b_2,\ b_2^{-1}b_3,\ \ldots\ ,\ b_{n-1}^{-1}b_n,\ b_n^{-1}a,\ a^2\}
\]
and
\[
T_2=
\{b_1,\ b_1^{-1}b_2^{-1},\  b_2b_3,\ b_3^{-1}b_4^{-1},\ \ldots\ ,\ b_{n-2}b_{n-1},\ b_{n-1}^{-1}b_n^{-1},\ b_n^3a, \ b_n^2a\}
\]
if $n$ is even, else
\[
T_2=
\{b_1^{-1}, b_1b_2, b_2^{-1}b_3^{-1},\ \ldots\ , b_{n-2}b_{n-1}, b_{n-1}^{-1}b_n^{-1}, b_n^3a, \ b_n^2a\}
\]
if $n$ is odd, we claim that these sets yield a ramification structure for $G_k$. 
We will only consider the case where $n$ is even. In the case that $n$ is odd, the same arguments apply.

As in the proof of Theorem~\ref{multi-GGS main}, if $x=a^2 \in T_1$ with $y\in T_2$, then obviously condition~\eqref{intersection1} holds.
We now let $x\in T_1\backslash \{a^2\}$ with $y=b_1 \in T_2$. 
Then if $x$ has order~$p^2$, by Lemmas~\ref{lem:multi-GGS} and \ref{EGS-single symmetric vector}, 
we have $x^p \in \St_{G_k}(2)$ while $b_1$ is of order $p$ and $b_1 \notin \St_{G_k}(2)$. Hence we suppose that $x$ has order~$p$, and therefore $x\in\{b_1^{-1}b_2, b_2^{-1}b_3, \ldots , b_{n-1}^{-1}b_n\}$. We know from the proof of Lemma~\ref{EGS-single symmetric vector} that for such an~$x$, there are two components of~$\phi_k(x)$ having non-trivial directed automorphisms, whereas $\phi_k(b_1)$ has only one such component.
Thus, the claim follows in these cases.

Next, we let  $x\in \{ab_1, b_n^{-1}a\}\subseteq T_1$. 
If $y\in \{ b_n^3a, b_n^2a\}\subseteq T_2$, then we apply Proposition~\ref{EGS-key}.
On the other hand, if $y\in  \{b_1^{-1}b_2^{-1},\  b_2b_3, \ \ldots\ ,\ b_{n-1}^{-1}b_n^{-1}\}$, we know that either $y^p\in \St_{G_k}(3)\backslash \St_{G_k}(4)$ by  Lemma~\ref{EGS-single symmetric vector}, or $y$ has order~$p$ and $y\in \St_{G_k}(1)\backslash \St_{G_k}(2)$.
However $x^p \in \St_{G_k}(2)\backslash \St_{G_k}(3)$ by Lemma~\ref{lem:multi-GGS}. 
Hence, we are done in these cases. 
The same argument applies if $x\in \{b_1^{-1}b_2,\ b_2^{-1}b_3, \ \ldots\ ,\ b_{n-1}^{-1}b_n\} \subseteq T_1$ with $y\in \{ b_n^3a, b_n^2a\}\subseteq T_2$.

Thus, it remains to check the cases $x\in \{b_1^{-1}b_2,\  b_2^{-1}b_3, \ \ldots\ ,\ b_{n-1}^{-1}b_n\} \subseteq T_1$ and  $y\in  \{b_1^{-1}b_2^{-1},\  b_2b_3,  \ \ldots\ ,\ b_{n-1}^{-1}b_n^{-1}\} \subseteq T_2$.
If exactly one of $x$ and $y$ has order~$p$, then clearly  \eqref{intersection1} holds from what we have seen above. Suppose next  that both $x$ and $y$ have order~$p$. From the proof of Lemma~\ref{EGS-single symmetric vector}, we have
\[
\phi_k(x)
=\big( a^*, \ldots, a^*,  b_{i+1},  a^*, \ldots, a^*,  b_i^{-1},  a^*, \ldots, a^*\big)
\]
and
\[
\phi_k(y)
=\big( a^*, \ldots, a^*,  b_{\ell+1},  a^*, \ldots, a^*,  b_\ell,  a^*, \ldots, a^*\big).
\]
If there was $g\in G_k$ such that $x^{\delta}=(y^\epsilon)^g$ for some $\epsilon,\delta\in\mathbb{F}_p^*$, 
first of all we must have a compatible spacing of the directed automorphisms in $\phi_k(x)$ as compared to $\phi_k(y)$; that is,
$$
\{j_{i+1}-j_i, p-j_{i+1}+j_i-2 \}=\{j_{\ell+1}-j_\ell, p-j_{\ell+1}+j_\ell-2\}.
$$
Without loss of generality, suppose that $j_{i+1}-j_i=j_{\ell+1}-j_\ell$; the other case is similar. It follows that $g=a^{j_{\ell+1}-j_{i+1}}h$, for some $h\in\text{St}_{G_k}(1)$. Writing $\phi_k(h)=(h_1,\ldots,h_p)$, from the fact that 
$$
\phi_{k-1}\big((b_{\ell+1}^\epsilon)^{h_{p-j_{i+1}+1}}\big)=\phi_{k-1}(b_{i+1}^\delta),
$$
it similarly follows that $h_{p-j_{i+1}+1}=a^{j_{\ell+1}-j_{i+1}}\widetilde{h}$ for some $\widetilde{h}\in\text{St}_{G_{k-1}}(1)$;
recall also that $k\ge 5$. Then, since there is only one defining vector, we deduce that if $\langle x\rangle \cap \langle y\rangle^g\ne \{1\}$, we must have  $\epsilon=\delta$, equivalently $y^g=x$. However, now considering the fact that 
$$
\phi_{k-1}\big(b_{\ell}^{h_{p-j_i+1}}\big)=\phi_{k-1}(b_i^{-1})
$$
yields $h_{p-j_i+1}=a^{j_{\ell}-j_{i}}\widehat{h}$ for some $\widehat{h}\in\text{St}_{G_{k-1}}(1)$. Then  
as
\[
\phi_{k-1}(b_i^{-1})=(a^{-e_i},\ldots, a^{-e_{p-1}},b_i^{-1}, a^{-e_1},\ldots, a^{-e_{i-1}})
\]
and
\[
\phi_{k-1}(b_{\ell}^{a^{j_{\ell}-j_{i}}\widehat{h}})\equiv(a^{e_i},\ldots, a^{e_{p-1}},b_{\ell}, a^{e_1},\ldots, a^{e_{i-1}})\pmod {G_{k-2}'\times\cdots\times G_{k-2}'},
\]
it follows that
$\mathbf{e}=-\mathbf{e}$, a contradiction.

Lastly, suppose that both $x$ and $y$ have order~$p^2$. By the proof of Lemma~\ref{EGS-single symmetric vector},
\[
\phi_k(x^p)=
\big(1, \ldots, 1, (a^{-e_{j_{i+1}-j_i}}b_{i+1})^p, \ 1,  \ldots, 1, \ \big( (a^{e_{j_{i+1}-j_i}}b_i^{-1})^p \big)^a, 1,  \ldots, 1  \big).
\]
Now, since $b_{\ell}^{-1}b_{\ell+1}^{-1}=\big( (b_{\ell}b_{\ell+1})^{b_{\ell}}\big)^{-1}$, it is enough to consider only $y=b_{\ell}b_{\ell+1}$.
Then 
\[
\phi_k(y^p)=
\big(1, \ldots, 1, (a^{e_{j_{\ell+1}-j_\ell}}b_{\ell+1})^p, \ 1,  \ldots, 1, \ \big( (a^{e_{j_{\ell+1}-j_\ell}}b_{\ell})^p \big)^a, 1 \ldots, 1  \big).
\]
We now compare the third-level decomposition of $x^p$ versus $y^p$. As seen in the above case, after aligning all the directed generators at the third-level, we see that in the third-level decomposition for~$x^p$, associated to $p$ of the directed generators (those descending from $(a^{-e_{j_{i+1}-j_i}}b_{i+1})^p$), we have $\mathbf{e}$ as the defining vector, and we have $-\mathbf{e}$ associated to the remaining $p$ directed generators (those descending from $(a^{e_{j_{i+1}-j_i}}b_i^{-1})^p$). However, in the third-level decomposition for~$y^p$, all associated defining vectors are~$\mathbf{e}$. Hence we
conclude that 
$\langle x^p\rangle \cap \langle y^p\rangle^g=\{1\}$ for any $g\in G_k$.

 This completes the periodic case.

\vspace{10pt}
We now suppose that $G$ is not periodic.
We fix $k\ge 3$.
If $n$ is even, the sets
\[
\{ab_1b_2^{-1},\ b_2,\ b_1^{-1},\ b_2b_3^{-1}a,\ a^{-1} b_3b_4^{-1},\ \ldots\ ,\ a^{-1}b_{n-1}b_n^{-1}\}
\]
and 
\[
\{ab_1,b_2a^{-1},\ ab_3,b_4a^{-1},\ \ldots\ ,\ ab_{n-1},\ b_n^\lambda a^{-1},\ a\},
\]
where $\lambda\in\mathbb{F}_p^*$ is such that $n-1+\lambda\not\equiv 0 \pmod p$,
are both systems of generators for~$G_k$, where if $n=2$, instead of the former set we  consider $\{ab_1b_2^{-1},\ b_1,\ b_2^{-1}\}$.
If $n$ is odd, we correspondingly have
\[
\{ab_1b_2^{-1},\ b_2,\ b_1^{-1},\ b_2b_3^{-1}a,\ a^{-1} b_3b_4^{-1},\ \ldots\ ,\ b_{n-1}b_n^{-1}a\} 
\]
and
\[
\{ab_1,\ b_2a^{-1},\ ab_3,\ b_4a^{-1},\ \ldots\ ,\ b_{n-1}a^{-1},\ ab_n^\lambda ,\ a\}.
\]

We claim that these sets yield a ramification structure for~$G_k$.  For convenience, we suppose that $n$ is even; the odd case follows analogously.

Set 
\begin{align*}
T_1&=\{ab_1b_2^{-1},\ b_2,\ b_1^{-1},\ b_2b_3^{-1}a,\ a^{-1} b_3b_4^{-1},\ \ldots\ ,\ b_{n-2}b_{n-1}^{-1}a,\ a^{-1}b_{n-1}b_n^{-1},\ ab_2b_n^{-1} \}
\end{align*}
 if $n>2$, else
 \[
 T_1=\{ab_1b_2^{-1},\ b_1,\ b_2^{-1},\   ab_1b_2^{-1}b_1b_2^{-1} \},
 \]
 together with
\[
T_2=\{ab_1,\ b_2a^{-1},\ \ldots\ ,\ ab_{n-1},\ b_n^\lambda a^{-1} ,\ a,\ ab_1b_2\cdots b_{n-1}b_n^{\lambda}\}.
\]

Clearly for all $x\in T_1$ with $y=a$ the equation \eqref{intersection1} holds. Similarly for $x\in \{b_1^{\pm 1}, b_2^{\pm 1}\}\cap T_1$ with $y\in T_2$.

For the remaining cases, we first  consider the case $x=ab_1b_2^{-1}$ with $y=ab_1$.
Assume that the symmetric defining vector $\mathbf{e}$ has exponent sum $s \not\equiv 0 \pmod p$.

Recall that
\[
\phi_k\Big(\big((ab_1b_2^{-1})^p\big)^g\Big)
=\Big((b_2^{-1}b_1^{a^*})^{g_1},\ldots,(b_2^{-1}b_1^{a^*}\big)^{g_{p}}\Big)
\]
for  some $g_1,\ldots,g_p\in G_{k-1}$ and unspecified exponents~$*$, and note that $\phi_{k-1}(b_2^{-1}b_1^{a^*})$ has at most two components with non-rooted automorphisms.
However
\[
\phi_k((ab_1)^{p^2})=\Big(\big((a^sb_1)^{a^*}\big)^p,\ldots,\big((a^sb_1)^{a^*}\big)^p\Big),
\]
and
\[
\phi_{k-1} \big((a^sb_1)^p\big) = \big((a^sb_1)^{a^*},\ldots,(a^sb_1)^{a^*}\big).
\]
From comparing the second-level decomposition of $(ab_1)^{p^2}$ with the second-level decomposition of $(ab_1b_2^{-1})^p$, we see that $\langle (ab_1b_2^{-1})^p\rangle^g$, for any $g\in G_k$, has trivial intersection with $\langle (ab_1)^{p^{k-1}}\rangle$.

For the other cases, a similar argument follows. Indeed, consider in particular the last element in~$T_2$. We observe that modulo $G_{k-1}'\times \cdots\times G_{k-1}'$,
\begin{align*}
&\phi_k\big((ab_1b_2\cdots b_{n-1}b_n^{\lambda})^p\big)\equiv
\big(\, a^{s(n-1+\lambda)} b_1b_2\cdots b_{n-1}b_n^\lambda\,,\,\ldots\,,\,   a^{s(n-1+\lambda)} b_1b_2\cdots b_{n-1}b_n^\lambda\,\big).
\end{align*}
From the choice of~$\lambda$, we have that $s(n-1+\lambda)\not\equiv 0\pmod p$, and the above argument applies.

The proof is now complete.
\end{proof}

\subsection*{Acknowledgements}
We thank G.\,A. Fern\'{a}ndez-Alcober for his useful comments.

\end{document}